\documentclass[a4paper,12pt]{amsart}
\usepackage{amssymb,amsmath,amsfonts,amsthm}
\usepackage[all]{xy}
\usepackage{enumerate}
\usepackage{mathrsfs}
\usepackage{graphicx}

\newcommand{\N}{\mathbb{N}}
\newcommand{\PP}{\mathbb{P}}
\newcommand{\VI}{\mathcal{B}}

\newcommand{\point}{\mathfrak{p}}
\newcommand{\RI}{\mathcal{R_I}}
\newcommand{\SI}{\mathcal{S_I}}

\newcommand{\Y}{\mathcal{Y}}

\newcommand{\rank}{{\operatorname{rank}}}
\newcommand{\Ker}{{\operatorname{Ker}}}

\newcommand{\sat}{{\operatorname{sat}}}

\newcommand{\indeg}{{\operatorname{indeg}}}

\newcommand{\Proj}{{\operatorname{Proj}}}

\newcommand{\Sym}{{\operatorname{Sym}}}
\newcommand{\Rees}{{\operatorname{Rees}}}

\newcommand{\Syz}{{\operatorname{Syz}}}
\newcommand{\chara}{{\operatorname{char}}}
\newcommand{\pmt}[1]{\begin{pmatrix}#1\end{pmatrix}}

\def\ff{{\bf f}} 
\def\gg{{\bf g}} 
\def\XX{{\bf X}} 
\def\TT{{\bf T}}
\numberwithin{equation}{section} 
\newtheorem{pro}{Proposition}[section]
\newtheorem{Lem}[pro]{Lemma}
\newtheorem{Cor}[pro]{Corollary} 
\newtheorem{Theo}[pro]{Theorem}

\newtheorem*{Theoetoile}{Theorem} % théorème non numéroté
\newtheorem*{Corollaryetoile}{Corollary} % conjecture non numérotée
\theoremstyle{definition}

\newtheorem{Exem}[pro]{Example}

\newtheorem{Rem}[pro]{Remark} 
\let\epsilon\varepsilon
\let\kappa=\varkappa

\usepackage{hyperref}
\hypersetup{
	colorlinks=true,       % false: boxed links; true: colored links
	linkcolor=red,          % color of internal links
	citecolor=blue,        % color of links to bibliography
	filecolor=magenta,      % color of file links
	urlcolor=cyan           % color of external links
}

\title{Fibers of rational maps and Jacobian matrices}
\date{\today}

\author{Marc Chardin}
\address{Institut de Math\'{e}matiques de Jussieu. UPMC,  4 place Jussieu, 75005 Paris, France}
\email{marc.chardin@imj-prg.fr}
\author{Steven Dale Cutkosky}
\thanks{Steven Dale Cutkosky was partially supported by NSF}

\address{Department of Mathematics, University of Missouri, Columbia, MO 65211, USA}
\email{cutkoskys@missouri.edu}

\author{Quang Hoa Tran}
\address{University of Education, Hue University,  34 Le Loi St., Hue City, Vietnam
\& Institut de Math\'{e}matiques de Jussieu. UPMC, 4 place Jussieu, 75005 Paris, France}
\email{quang-hoa.tran@imj-prg.fr}
\begin{document}
\maketitle
\begin{abstract}
A rational map $\phi: \PP_k^m \dashrightarrow \PP_k^n$ is defined by homogeneous polynomials of a common degree $d$. We establish a linear bound in terms of $d$ for the number of $(m-1)$-dimensional fibers of $\phi$, by using ideals of minors of the Jacobian matrix. In particular, we answer affirmatively   Question~11 in \cite{QHTran17}. 
\end{abstract}

\section{Introduction}
Let $k$ be a field and $\phi: \PP_k^m\dashrightarrow \PP_k^n$ be a rational map. Such a map $\phi$ is defined by homogeneous polynomials $f_0,\ldots,f_n,$ of the same degree $d,$ in a standard graded polynomial  ring $R=k[X_0,\ldots,X_m],$ such that $\gcd(f_0,\ldots,f_n)=1.$ The ideal $I$ of $ R$ generated by these polynomials is called the \textit{base ideal of $\phi$}.  The scheme $\VI:=\Proj(R/I)\subset \PP_k^m $ is  called the \textit{base locus of $\phi$}. Let $B=k[T_0,\ldots,T_n]$ be the homogeneous coordinate ring of $\PP_k^n.$ The map $\phi$ corresponds to the $k$-algebra homomorphism $\varphi: B\longrightarrow R,$ which sends each $T_i$ to $f_i.$ Then the kernel of this homomorphism defines the closed image $\mathscr{S}$  of $\phi.$ In other words,  after degree renormalization, $k[f_0,\ldots,f_n]\simeq B/\Ker(\varphi)$ is the homogeneous coordinate ring of $\mathscr{S}.$  The minimal set of generators of $\Ker(\varphi)$ is called its \textit{implicit equations} and the \textit{implicitization problem} is to find these implicit equations.

The implicitization problem for  curves or surfaces has been of increasing interest to commutative algebraists and algebraic geometers due to its applications in Computer Aided Geometric Design as explained by Cox \cite{Cox05}. 

We blow up the base locus of $\phi$ and obtain the following commutative diagram
\begin{displaymath}
\xymatrix{ \Gamma \ar@{^{(}->}[rr]\ar[d]_{\pi_1}&& \PP_k^m\times \PP_k^n \ar[d]^{\pi_2} \\ \PP_k^m \ar@{-->}[rr]^{\phi} &&\PP_k^n.     }
\end{displaymath}
The variety $\Gamma$ is the blow-up of $\PP_k^m$ at $\VI$ and it is also the Zariski closure of the graph of $\phi$ in $\PP^m_k\times \PP_k^n.$ Moreover, $\Gamma$ is the geometric version of the Rees algebra of $I,$ i.e. $\Proj(\RI)=\Gamma.$ As $\RI$ is the graded domain defining $\Gamma,$ the projection  $\pi_2(\Gamma)=\mathscr{S}$  is defined by the graded domain $\RI\cap k[T_0,\ldots,T_n]$ and we can thus obtain the implicit equations of $\mathscr{S}$  from the defining equations of $\RI.$

In geometric modeling,  it is of vital importance to have a detailed knowledge of the geometry of the objects and of the parametric representations one is working with.  The question of how many times the same point is being painted (i.e. corresponds to distinct values of parameter) depends not only on the variety itself, but also on the parameterization.  It is of interest to determine the singularities of the parameterizations, in particular their fibers.  More precisely, we set
$$\pi:={\pi_2}_{\mid \Gamma}: \Gamma \longrightarrow \PP_k^n.$$ 
For every closed point $y\in \PP_k^n,$ we will denote by $k(y)$ its residue field. If $k$ is assumed to be  algebraically closed, then $k(y)\simeq k.$ The fiber of $\pi$  at  $y\in \PP_k^n$ is the subscheme
\begin{align*}
\pi^{-1}(y)=\Proj(\RI\otimes_B k(y)) \subset \PP_{k(y)}^m\simeq \PP_k^m.
\end{align*}
Suppose that  $m\geq 2$ and $\phi$ is generically finite onto its image.  Then the set
$$\Y_{m-1}=\{y\in \PP_k^n \mid \dim  \pi^{-1}(y)=m-1\}$$ 
consists of only  a finite number of points in $\PP_k^n.$  For each $y\in \Y_{m-1},\, \pi^{-1}(y)$ is a $(m-1)$-dimensional subcheme of $\PP_k^m$ and thus the unmixed component of maximal dimension is defined by a homogeneous polynomial $h_y\in R.$ Our main purpose is to establish a bound for  $\sum_{y\in \Y_{m-1}}\deg(h_y)$ in terms of the degree $d.$ A quadratic bound in $d$ for this sum of  one-dimensional fibers of a parameterization surface $\phi: \PP_k^2 \dashrightarrow \PP_k^3$ is given by the third named author \cite{QHTran17}. More precisely, he proved the following.
\begin{Theoetoile}\cite[Theorem 7 \& 9]{QHTran17}
Let $I$ be a homogeneous ideal of $R=k[X_0,X_1,X_2]$ generated by a minimal generating set of homogeneous polynomials $f_0,\ldots,f_3$ of degree $d.$ Suppose that  $I$ has codimension 2 and $\VI=\Proj(R/I)$ is locally a complete intersection of dimension zero. Let $I^\sat:=I\colon_R(X_0,X_1,X_2)^\infty$ be the saturation of $I$ and $\mu=\inf \{\nu \mid I_\nu^\sat\neq 0\}.$ 
\begin{enumerate}
\item[\em (i)] If $\mu <d,$ then 
$\sum\limits_{y\in \Y_1}\deg(h_y)\leq \mu.$
\item[\em (ii)] If $\mu=d$, then
\begin{displaymath}
\sum_{y\in \Y_1}\deg(h_y)\leq \begin{cases} 
4  & \text{if}  \quad d=3,\\
\left\lfloor \frac{d}{2} \right \rfloor  d-1 &\text{if} \quad d\geq 4.
\end{cases}
\end{displaymath}
\end{enumerate}
\end{Theoetoile}

In this paper, we refine and generalize the above theorem. Recall that if  $\ff:=f_0,\ldots,f_n$ are polynomials in $R=k[X_0,\ldots,X_m],$ then the Jacobian matrix of $\ff$ is defined by
$$
J(\ff)=\left(\begin{array}{ccc}
\frac{\partial f_0}{\partial X_0}&\cdots& \frac{\partial f_0}{\partial X_m}\\
\vdots & &\vdots\\
\frac{\partial f_n}{\partial X_0}&\cdots&\frac{\partial f_n}{\partial X_m}
\end{array}
\right). 
$$
Denote by $I_s(J(\ff))$ the ideal of $R$ generated by the $s$-minors of $J(\ff),$ where $1\leq s\leq \min\{m+1,n+1\}.$ Our main result is the following.

\begin{Theoetoile}[Theorem~\ref{Theorem2.5}]
Let $I$ be a homogeneous ideal of $R=k[X_0,\ldots,X_m]$ generated by a minimal generating set of homogeneous polynomials $\ff:=f_0,\ldots,f_n$ of degree $d.$ Suppose that  $\gcd(f_0,\ldots,f_n)=1$ and $I_3(J(\ff))\neq 0.$ Let $F$ be  the greatest common divisor  of generators of $I_3(J(\ff)).$  Then
$$\sum_{y \in \Y_{m-1}}\deg(h_y)\leq \sum_{y \in \Y_{m-1}}\sum_{i=1}^{r_y}(2e_{i}-1)\deg(h_{i})\leq \deg(F)\leq 3(d-1),$$
where $h_y~=~h_1^{e_1}\cdots h_{r_y}^{e_{r_y}}$ is an irreducible factorization of $h_y$ in $R.$ 
\end{Theoetoile}
If  the field $k$ is of characteristic zero, then the assumption $I_3(J(\ff))\neq 0$ is always satisfied, due to the hypothesis that $\phi$ is generically finite onto its image. Therefore, the above theorem is a significant  improvement and  generalization of  the results in \cite{QHTran17}, at least in the case where $k$ is of characteristic zero.

\smallskip
For a parameterization of surfaces $\phi: \PP_k^2 \dashrightarrow \PP_k^3,$  we give a linear bound for $\sum_{y\in \Y_1}\deg(h_y)$ which answers affirmatively  Question~11 in \cite{QHTran17}. 
\begin{Corollaryetoile}[Corollary~\ref{Corollary4.4}]
Let $I$ be a homogeneous ideal of $R=k[X_0,\ldots,X_m]$ generated by a minimal generating set of homogeneous polynomials $\ff:=f_0,\ldots,f_3$ of degree $d.$ Suppose that  $I$ has codimension 2. Assume further  that the characteristic of $k$ does not divide $d$ and $[k(\ff):k(\XX)]$ is separable.  Then
$$\sum_{y\in \Y_1}\deg(h_y)\leq 3(d-1)-\indeg(\Syz(I))< 3(d-1).$$
\end{Corollaryetoile}
Observe that the last two conditions in the above corollary are automatically satisfied if $k$ is of characteristic zero.

\section{Tangent space maps and Jacobian matrices}
Suppose that $X$ is a $k$-scheme where $k$ is an algebraically closed field, and $q\in X$ is a closed point. The tangent space $T(X)_q$ of $X$ at $q$ is the $k$-vector space 
$$T(X)_q={\rm Hom}_k(m_q/m_q^2,k)$$
 where $m_q$ is the maximal ideal of $\mathcal O_{X,q}$.
Suppose that $Y$ is another $k$-scheme and $\phi:X\rightarrow Y$ is a morphism of $k$-schemes.
Then $\phi^*:\mathcal O_{Y,\phi(q)}\rightarrow \mathcal O_{X,q}$ induces a homomorphism of $k$-vector spaces 
$d\phi_q:T(X)_q\rightarrow T(Y)_{\phi(q)}$. If $V$ is a subscheme of $X$ and $W$ is a subscheme of $Y$ such that $\phi(V)\subset W$, then we have a natural commutative diagram of homomorphisms of $k$-vector spaces
\begin{equation}\label{eq1}
\begin{array}{ccc}
T(V)_q&\subset&T(X)_q\\
\downarrow&&\downarrow\\
T(W)_{\phi(q)}&\subset&T(Y)_{\phi(q)}.
\end{array}
\end{equation}

From now on, we will consider the following situation. Suppose that $k$ is an algebraically closed field of characteristic $p\geq 0$. Consider  $f_0,\ldots,f_n$ homogeneous polynomials of a common degree $d$ in the standard polynomial ring $R:=k[X_0,\ldots,X_m],$ such that $\gcd(f_0,\ldots,f_n)=1$. Let $\phi:\PP^m_k\dashrightarrow \PP^n_k$ be a rational map defined by $f_0,\ldots,f_n.$   The maximal open set on which $\phi$ is a morphism is $\Omega_\phi=\PP_k^m\setminus Z(f_0,\ldots,f_n)$. Let 
$$
J(\ff)=\left(\begin{array}{ccc}
\frac{\partial f_0}{\partial X_0}&\cdots& \frac{\partial f_0}{\partial X_m}\\
\vdots & &\vdots\\
\frac{\partial f_n}{\partial X_0}&\cdots&\frac{\partial f_n}{\partial X_m}
\end{array}
\right)
$$
be the Jacobian matrix of $\ff=f_0,\ldots,f_n.$ For any closed point $q=(q_0:\cdots: q_m)\in \PP_k^m$, we denote by $J(q)$ the matrix obtained from $J(\ff)$ by mapping $X_i$ to $q_i$ for all $i=0,\ldots,m.$
The entries of this matrix are defined by $q$ up to multiplication by a common non zero scalar.
\begin{pro} \label{Lemma1} 
Suppose that $p$ does not divide $d$ and $q\in \Omega_\phi$ is a closed point. Then
$$
\mbox{rank }J(q)=\mbox{rank }d\phi_q+1,
$$
where $d\phi_q:T(\PP_k^m)_q\rightarrow T(\PP_k^n)_{\phi(q)}$ is the tangent space map.
\end{pro}

\begin{proof} 
After possibly making linear changes of homogeneous coordinates  in $\PP_k^m$ and $\PP_k^n$, we may assume that $q=(1:0:\cdots:0)$ and $\phi(q)=(1:0:\cdots:0)$. Let $\overline X_i=\frac{X_i}{X_0}$ for $1\le i\le m$. Let $F_i=\frac{f_i}{X_0^d}\in k[\overline X_1,\ldots,\overline X_m]$, which is the affine coordinate ring of $\PP_k^m\setminus Z(X_0)$.
As $\phi$ is a regular map near $q,$ 
$$
\phi=(f_0:f_1:\cdots:f_n)=(1:g_1:\cdots:g_n),
$$
where $g_i=\frac{f_i}{f_0}=\frac{F_i}{F_0}$. Let $\alpha=(1,0,\ldots,0)$.  We have that 
$$
\frac{\partial f_j}{\partial X_0}(\alpha)=df_j(\alpha)
$$
for all $j=0,\ldots,n,$ by Euler's formula. Thus
$$
\begin{array}{lll}
\mbox{rank }J(p)&=&\mbox{rank }\left(\begin{array}{cccc}
df_0(\alpha)&\frac{\partial f_0}{\partial X_1}(\alpha)&\cdots&\frac{\partial f_0}{\partial X_m}(\alpha)\\
df_1(\alpha)&\frac{\partial f_1}{\partial X_1}(\alpha)&\cdots&\frac{\partial f_1}{\partial X_m}(\alpha)\\
\vdots&\vdots&&\vdots\\
df_n(\alpha)&\frac{\partial f_n}{\partial X_1}(\alpha)&\cdots&\frac{\partial f_n}{\partial X_m}(\alpha)
\end{array}\right)\\
\\
	
&=&\mbox{rank}\left(\begin{array}{cccc}
df_0(\alpha)&\frac{\partial f_0}{\partial X_1}(\alpha)&\cdots&\frac{\partial f_0}{\partial X_m}(\alpha)\\
0&\frac{\partial f_1}{\partial X_1}(\alpha)&\cdots&\frac{\partial f_1}{\partial X_m}(\alpha)\\
\vdots&\vdots&&\vdots\\
0&\frac{\partial f_n}{\partial X_1}(\alpha)&\cdots&\frac{\partial f_n}{\partial X_m}(\alpha)
\end{array}\right)\\
\\
	
&=&\mbox{rank}\left(\begin{array}{ccc}
\frac{\partial f_1}{\partial X_1}(\alpha)&\cdots&\frac{\partial f_1}{\partial X_m}(\alpha)\\
\vdots&&\vdots\\
\frac{\partial f_n}{\partial X_1}(\alpha)&\cdots&\frac{\partial f_n}{\partial X_m}(\alpha)
\end{array}\right)+1
\end{array}
$$

Let $\overline\alpha=(0,\ldots,0)$. As
$$
\frac{\partial f_i}{\partial X_j}(\alpha)=\frac{\partial F_i}{\partial\overline X_j}(\overline\alpha)
$$
for $1\le i\le n$ and $1\le j\le m$ (it suffices to check this on a monomial), so 
$$
\begin{array}{lll}
\mbox{rank}\left(\begin{array}{ccc}
\frac{\partial f_1}{\partial X_1}(\alpha)&\cdots&\frac{\partial f_1}{\partial X_m}(\alpha)\\
\vdots&&\vdots\\
\frac{\partial f_n}{\partial X_1}(\alpha)&\cdots&\frac{\partial f_n}{\partial X_m}(\alpha)
\end{array}\right)
&=&
\mbox{rank}\left(\begin{array}{ccc}
\frac{\partial F_1}{\partial \overline X_1}(\overline\alpha)&\cdots&\frac{\partial F_1}{\partial \overline X_m}(\overline\alpha)\\
\vdots&&\vdots\\
\frac{\partial F_n}{\partial \overline X_1}(\overline\alpha)&\cdots&\frac{\partial F_n}{\partial \overline X_m}(\overline\alpha)
\end{array}\right)\\
\\
	
&=&\mbox{rank}\left(\begin{array}{ccc}
\frac{\partial g_1}{\partial \overline X_1}(\overline\alpha)&\cdots&\frac{\partial g_1}{\partial \overline X_m}(\overline\alpha)\\
\vdots&&\vdots\\
\frac{\partial g_n}{\partial \overline X_1}(\overline\alpha)&\cdots&\frac{\partial g_n}{\partial \overline X_m}(\overline\alpha)
\end{array}\right)

\end{array}
$$
since
$$
\frac{\partial}{\partial \overline X_j}\left(\frac{F_i}{F_0}\right)=\frac{\partial F_i}{\partial \overline X_j}F_0^{-1}-F_0^{-2}\frac{\partial F_0}{\partial \overline X_j}F_i,
$$
so
$$
\frac{\partial g_i}{\partial \overline X_j}(\overline\alpha)=\frac{1}{F_0(\overline\alpha)}\frac{\partial F_i}{\partial \overline X_j}(\overline\alpha)
$$
for $1\le i\le n$ and $1\le j\le m$, as $F_0(\overline\alpha)\ne 0$ and $F_i(\overline\alpha)=0$ for $1\le i\le n$.
\end{proof}

\begin{Rem}\label{Remark2} 
If $p$ divides $d$ the proof of Proposition~\ref{Lemma1} shows that we can either have 
$$
\mbox{rank }J(q)=\mbox{rank }d\phi_q+1\mbox{ or }\mbox{rank }J(q)=\mbox{rank }d\phi_q.
$$
Both options are possible.
\end{Rem}
\begin{pro}\label{Proposition1.3} 
Suppose that $r\in \N$ and $V$ is a subvariety of $\PP_k^m$ such that $V\cap \Omega_\phi\neq \emptyset$ and $r=\dim V-\dim \phi(V)$. Then $V\subset Z(I_{m-r+2}(J(\ff)))$, where $I_{m-r+2}(J(\ff))$ is the ideal generated by the $(m-r+2)$-minors of $J(\ff)$.
\end{pro}

\begin{proof} 
There exists a dense open subset $U$ of $V$ such that for any  $q\in U$, $V$ is smooth at $q$  and $\phi(V)$ is smooth at $\phi(q)$ (take $U$ to be the intersection of smooth locus of $V$ with the preimage of the smooth locus of $\phi(V)$). We have that $\dim T(\phi(V))_{\phi(q)}=\dim V-r$ for $q\in U$, so by consideration of diagram \eqref{eq1}, it follows  that $\dim \Ker \ d\phi_q\ge r$ for $q\in U$,
hence $\rank\ d\phi_q\le m-r$ for $q\in U$. By Proposition~\ref{Lemma1} and Remark~\ref{Remark2},  we have that
$$
\rank \,J(q)\leq  \rank \,d\phi_q+1\le m-r+1
$$
for $q\in U$. Thus $U$ is contained in the closed set $Z(I_{m-r+2}(J(\ff)))$, so the closure $V$ of $U$ is contained in this set.
\end{proof}

%#########
\section{Bound for the number of $(m-1)$-dimensional fibers of a rational map $\phi: \PP_k^m \dashrightarrow \PP_k^n$ }
We adopt in this section the notations and  hypotheses of the introduction. 
Recall that  \begin{align*}
\phi:\quad &\PP_k^m  -\dashrightarrow \PP_k^n\\
& \; x \longmapsto (f_0(x):\cdots: f_n(x))
\end{align*}
is a rational map whose closed image is a subvariety $\mathscr{S}$ in $\PP_k^n$ and $\Gamma\subset \PP_k^m\times \PP_k^n$ is the closure of its graph. We have the following diagram
$$ \xymatrix@1{\Gamma\ \ar[d]_{\pi_1}\ar@{^(->}[r] & \PP_k^m\times \PP_k^n \ar[d]^{\pi_2}\\  \PP^m_k \ar@{-->}[r]_\phi& \PP^n_k.}$$

Furthermore, $\Gamma$ is the irreducible subscheme of $\PP_k^m\times \PP_k^n$ defined by the Rees algebra $\RI:=\Rees_R(I)$ (see \cite[Chapter II, \S 7]{Hartshorne77}).  Let $B:=k[T_0,\ldots,T_n]$ be  the homogeneous coordinate ring of $\PP_k^n$ and  $S:=R\otimes_k B= R[T_0,\ldots,T_n]$ with the standard bigraded structure by the canonical grading $\deg(X_i)=(1,0)$ and $\deg(T_j)=(0,1)$ for all $i=0,\ldots,m$ and $j=0,\ldots,n$. The natural bigraded morphism of bigraded $k$-algebras
\begin{align*}
\alpha:\quad & S \longrightarrow  \RI=\oplus_{s\geq 0} I(d)^s=\oplus_{s\geq 0} I^s(sd)\\
&  T_i \longmapsto f_i
\end{align*}
is onto and corresponds to the embedding $\Gamma \subset \PP_k^m\times \PP_k^n$.

Let $\mathfrak{P}$ be the kernel of $\alpha$. Then it is a homogeneous ideal of $S$ and the part of degree one of  $\mathfrak{P}$ in $T_i$, denoted by  $\mathfrak{P}_1=\mathfrak{P}_{(\ast, 1)},$  is the module of syzygies of the  $f_i$
$$a_0T_0 +\cdots+a_nT_n \in \mathfrak{P}_1\Longleftrightarrow a_0f_0+\cdots+a_nf_n=0.$$

Set $\SI:=\Sym_R(I)$ for the symmetric algebra of $I$. The  natural bigraded epimorphisms
\begin{align*}
S\longrightarrow S/(\mathfrak{P}_1)\simeq \SI \text{\quad and\quad} \SI\simeq S/(\mathfrak{P}_1)\longrightarrow S/\mathfrak{P}\simeq \RI
\end{align*}
correspond to the embeddings of schemes $ \Gamma \subset V\subset \PP_k^m\times \PP_k^n$
where $V$ is the projective scheme defined by $\SI$.

As the construction of symmetric algebras and Rees algebras commute with
localization, and both algebras are the quotient of a polynomial extension of the
base ring by the Koszul syzygies on a minimal set of generators in the case of a
complete intersection ideal, it follows that $\Gamma$ and $V$ coincide on $(\PP_k^m\setminus X) \times \PP_k^n,$ where $X$ is the (possibly empty) set of points where $\VI$ is not locally a complete intersection.

Now we set $\pi:={\pi_2}_{\mid \Gamma}: \Gamma \longrightarrow \PP_k^n.$ For every closed point $y\in \PP_k^n,$ we will denote by $k(y)$ its residue field, that is, $k(y)=(B_\point/\point B_\point)_0,$ where $\point$ is the defining prime ideal of $y.$ As $k$ is algebraically closed, $k(y)\simeq k.$ The fiber of $\pi$  at  $y\in \PP_k^n$ is the subscheme
\begin{align*}
\pi^{-1}(y)=\Proj(\RI\otimes_B k(y)) \subset \PP_{k(y)}^m\simeq \PP_k^m.
\end{align*}

Let $0\leq \ell\leq m.$ We define 
$$\Y_\ell=\{y\in \PP_k^n \mid \dim  \pi^{-1}(y)=\ell\}\subset \PP_k^n.$$

We are interested in studying the structure of $\Y_\ell.$ First, Chevalley's theorem shows  that the subsets $\Y_\ell$ are constructible, that is, they can be written as 
$$\Y_\ell= \bigsqcup_{i=1}^s (U_i\cap Z_i),$$
where $U_i$ (respectively $Z_i$) are  open (respectively closed) subsets of $\PP_k^n.$

\begin{Lem} \label{Lemma2.1}
Let $\phi: \PP_k^m\dashrightarrow \PP_k^n$ be a rational map and $\Gamma$ be the closure of the graph of $\phi.$ Consider the canonical projection $\pi: \Gamma\longrightarrow \PP_k^n$. Then
$$\dim \ \overline{\Y_\ell}+\ell \leq m.$$
Furthermore, this inequality is strict for any $\ell >  m-\dim \ \mathscr{S},$ where  $\mathscr{S}$ is  the closed image of $\phi.$
\end{Lem}
\begin{proof}
Set $V_\ell:=\overline{\pi^{-1}(\Y_\ell)},$ a subvariety of $\Gamma.$ For the first statement
$$\dim \ \overline{\Y_\ell }+\ell= \dim \ V_\ell\leq \dim\ \Gamma =\dim\ \mathscr{S}\leq m.$$
Moreover, if $\dim \ \overline{\Y_\ell}+\ell =m,$ then $ \dim \ V_\ell =\dim\ \Gamma =m.$ It implies that $\overline{\Y_\ell}=\mathscr{S}$ and proves the second assertion.
\end{proof}
From now on, we will always assume  that $\phi$ is generically finite onto its image, or equivalently that the closed image $\mathscr{S}$ of $\phi$ is a subvariety  in $\PP_k^n:=\Proj(B),$ of dimension $m.$  Therefore, by Lemma~\ref{Lemma2.1}, $\dim \ \overline{\Y}_m<0,$ which shows that $\Y_m=\emptyset.$ This was noticed  in  \cite[Lemma~14]{Botbol-Buse_Chardin14}.  Now if $\ell=m-1\geq 1,$ as $m\geq 2,$ then  $\overline{\Y}_{m-1}$  consists of only a finite number of points in $\PP_k^n.$  In other words,  $\pi$ only has a finite number of $(m-1)$-dimensional fibers. 

For any $y\in \Y_{m-1},\, \pi^{-1}(y)$ is a subcheme of $\PP_{k(y)}^m\simeq \PP_k^m$ of dimension $m-1,$ as $k$ is algebraically closed. Thus the unmixed part of the fiber $\pi^{-1}(y)$ is defined by a homogeneous polynomial $h_y\in R,$ as $R$ is factorial. Our purpose is then to bound  $\sum_{y\in \Y_{m-1}}\deg(h_y)$ in terms of the degree $d.$ 

The fibers of $\pi$ are defined by specialization of the Rees algebra. However, Rees algebras are hard to study. Fortunately, the symmetric algebra of $I$ is easier to understand than $\RI$ and the fibers of $\pi$ are closely related to the fibers of $$\pi^\prime:={\pi_2}_{\mid V}: V \longrightarrow \PP_k^n.$$
Recall that  for any  closed point $y\in \PP_k^n,$ the fiber of $\pi^\prime$ at $y$ is the subscheme
\begin{align*}
{\pi^\prime}^{-1}(y)=\Proj(\SI\otimes_B k(y))\subset \PP_{k(y)}^m\simeq \PP_k^m.
\end{align*}  
We have the following lemma. Recall that $X$ is the (possibly empty) set of points where $\VI$ is not locally a complete intersection. 
\begin{Lem}\label{Lemma2.0}
 The fibers of $\pi$ and $\pi^\prime$ agree outside $X,$ hence  they have the same  $(m-1)$-dimensional components.
\end{Lem}
\begin{proof}
The first statement holds since $\Gamma$ and $V$ coincide on $(\PP_k^m\setminus X) \times \PP_k^n.$ Moreover, as $I$ is assumed to have  codimension  at least 2, $\dim \ \VI \leq m-2,$ showing that $\dim X\leq m-2.$ The second statement follows.
 \end{proof}

The following lemma is a simple generalization of \cite[Lemma~10]{Botbol-Buse_Chardin14}.
\begin{Lem} \label{Lemma2.3}
Let $I$ be a homogeneous ideal of $R$ generated by a minimal generating set of homogeneous polynomials $\ff:=f_0,\ldots,f_n$ of degree $d$ and  suppose that  $\gcd(f_0,\ldots,f_n)=1.$  Assume that the fiber of $\pi^\prime$ over a closed point $y$ with coordinates $(p_0: \cdots: p_n)$ is of dimension $m-1,$ and its unmixed components are defined by $h_y\in R.$ Let $\ell_y$ be a linear form in $\TT:=T_0,\ldots,T_n$ such that $\ell_y(p_0,\ldots,p_n)=1$ and set $\ell_i(\TT):=T_i-p_i\ell_y (\TT)\ (i=0,\ldots,n).$ Then, $h_y = \gcd(\ell_0(\ff),\ldots,\ell_n(\ff))$  and
$$I=\ell_y(\ff)+h_y(g_0,\ldots,g_n)$$
with $\ell_i(\ff)=h_y g_i$ and $\ell_y(g_0,\ldots,g_n)=0.$
\end{Lem}
\begin{proof}
The proof of this result goes along the same lines as in the proof of \cite[Lemma~10]{Botbol-Buse_Chardin14}.
\end{proof}

For $\ff=f_0,\ldots,f_n,$ set 
\begin{displaymath}
J(\ff)=\pmt{\frac{\partial f_0}{\partial X_0} &\cdots &\frac{\partial f_0}{\partial X_m}\\ \vdots &&\vdots\\ \frac{\partial f_n}{\partial X_0} &\cdots &\frac{\partial f_n}{\partial X_m}}
\end{displaymath}
for the Jacobian matrix  of $\ff$ and $I_s(J(\ff))$ for the ideal of $R$ generated by the $s\times s$ minors of $J(\ff),$ where $1\leq s\leq m+1.$ 

\begin{Lem}\label{Lemma2.4}
Suppose that $\dim_k I_d=n+1$ and let  $\ff=f_0,\ldots,f_n$ and $\gg=g_0,\ldots,g_n$ be two bases of $I_d.$ Then  $I_s(J(\ff))=I_s(J(\gg))),$ for any $s$.
\end{Lem}

\begin{proof}
Indeed, these are the Fitting ideals (with the same indices) of two matrices that are equal after change of basis over the base field. 
 \end{proof}

\smallskip
Recall that for any $y\in \Y_{m-1}$, we  denote by $h_y\in R$ a defining equation of the unmixed part of the fiber $\pi^{-1}(y)$ (recall that $k$ is algebraically closed and $R$ is factorial). Assume that $h_y~=~h_1^{e_1}\cdots h_{r_y}^{e_{r_y}}$ is an irreducible factorization of $h_y$ in $R.$
\begin{Theo} \label{Theorem2.5}
Let $I$ be a homogeneous ideal of $R$ generated by a minimal generating set of homogeneous polynomials $\ff:=f_0,\ldots,f_n$ of degree $d.$ Suppose that  $\gcd(f_0,\ldots,f_n)=1$ and $I_3(J(\ff))\neq 0.$ Let $F$ be  the greatest common divisor  of generators of $I_3(J(\ff)).$  Then
$$\sum_{y \in \Y_{m-1}}\deg(h_y)\leq \sum_{y \in \Y_{m-1}}\sum_{i=1}^{r_y}(2e_{i}-1)\deg(h_{i})\leq \deg(F)\leq 3(d-1).$$
\end{Theo}

\begin{proof}
By Lemma~\ref{Lemma2.0}, the unmixed components of $\pi^{-1}(\point) $ and ${\pi^\prime}^{-1}(\point)$ are the same for every closed point $\point\in \Y_{m-1}.$ By Lemma~\ref{Lemma2.3}, there exists a homogeneous polynomial $f\in I$ of degree $d$ such that, for any $\point\in \Y_{m-1}$
$$I=(f)+h_y (g_{1y}, \ldots,g_{ny}),$$
for some $ g_{1y}, \ldots,g_{ny}\in R$. 

\smallskip
The Jacobian matrix of $\ff^\prime=(f,h_y g_{1y},\ldots,h_y g_{ny})$ is
\begin{displaymath}
J(\ff^\prime)=\pmt{\frac{\partial f}{\partial X_0} &\cdots &\frac{\partial f}{\partial X_m}\\ h_y\frac{\partial g_{1y}}{\partial X_0}+g_{1y}\frac{\partial h_y}{\partial X_0} &\cdots &h_y\frac{\partial g_{1y}}{\partial X_m}+g_{1y}\frac{\partial h_y}{\partial X_m}\\ \vdots & &\vdots\\ h_y\frac{\partial g_{ny}}{\partial X_0}+g_{ny}\frac{\partial h_y}{\partial X_0} &\cdots &h_y\frac{\partial g_{ny}}{\partial X_m}+g_{ny}\frac{\partial h_y}{\partial X_m}}.
\end{displaymath}
 For all $j=0,\ldots, m$
$$\frac{\partial h_y}{\partial X_j}=\sum_{i=1}^{r_y}e_{i}\frac{h_y}{h_{i}}\frac{\partial h_{i}}{\partial X_j},$$
therefore the $i$-th row of $J(\ff^\prime)$ has a common factor $h_1^{e_1-1}\cdots h_{r_y}^{e_{r_y}-1},$ for all $i=2,\ldots, n+1.$  It follows that, for any subset $\mathcal{I}$ of $\{1,\ldots,n+1\}$ with $3$ elements and  a subset $\mathcal{J}$ of $\{1,\ldots,m+1\}$ with $3$ elements,
\begin{equation}\label{equation2.2}
[J(\ff^\prime)]_{\mathcal{I},\mathcal{J}}=h_1^{2(e_1-1)}\cdots h_{r_y}^{2(e_{r_y}-1)}[M]_{\mathcal{I},\mathcal{J}},
\end{equation}
where $[Q]_{\mathcal{I},\mathcal{J}}$ denotes the 3-minor of a $(n+1)\times (m+1)$-matrix $Q$ that corresponds to the rows with index in $\mathcal{I}$ and the columns with index in $\mathcal{J}$ and  $M$  is the $(n+1)\times(m+1)$-matrix 
\begin{displaymath}
\pmt{\frac{\partial f}{\partial X_0} &\cdots &\frac{\partial f}{\partial X_m}\\ \widehat{h}_y\frac{\partial g_{1y}}{\partial X_0}+g_{1y}\sigma_0 &\cdots &\widehat{h}_y\frac{\partial g_{1y}}{\partial X_m}+g_{1y}\sigma_m\\ \vdots & &\vdots\\ \widehat{h}_y\frac{\partial g_{ny}}{\partial X_0}+g_{ny}\sigma_0 &\cdots &\widehat{h}_y\frac{\partial g_{ny}}{\partial X_m}+g_{ny}\sigma_m},
\end{displaymath}
where $\widehat{h}_y=h_1\cdots h_{r_y}$ and $\sigma_j=\sum\limits_{i=1}^{r_y}e_i\frac{\widehat{h}_y}{h_i}\frac{\partial h_i}{\partial X_j},\ (j=0,\ldots ,m).$ Thus there is a homogeneous polynomial  $P$ such that $[M]_{\mathcal{I},\mathcal{J}}=\widehat{h}_y P +[N]_{\mathcal{I},\mathcal{J}},$ where $N$ is the $(n+1)\times (m+1)$-matrix
 \begin{displaymath}
 \pmt{\frac{\partial f}{\partial X_0} &\cdots &\frac{\partial f}{\partial X_m}\\ g_{1y}\sigma_0 &\cdots&g_{1y}\sigma_m\\ \vdots &&\vdots\\ g_{ny}\sigma_0 &\cdots &g_{ny}\sigma_m}
 \end{displaymath}
which shows that $[N]_{\mathcal{I},\mathcal{J}}=0,$ as $\rank\ N\leq 2.$  By \eqref{equation2.2}, we obtain
\begin{equation} \label{equation2.3}
[J(\ff^\prime)]_{\mathcal{I},\mathcal{J}}=h_1^{2(e_1-1)}\cdots h_{r_y}^{2(e_{r_y}-1)}\widehat{h}_y P=h_1^{2e_1-1}\cdots h_{r_y}^{2e_{r_y}-1}P,
\end{equation}
for all $\mathcal{I},\mathcal{J}.$
Let $G$ be the greatest common divisor of generators of $I_3(J(\ff^\prime)).$ Then  $h_1^{2e_1-1}\cdots h_{r_y}^{2e_{r_y}-1}$ is a divisor of $G$ by \eqref{equation2.3}. By Lemma~\ref{Lemma2.4}, $h_1^{2e_1-1}\cdots h_{r_y}^{2e_{r_y}-1}$ is a divisor of $F.$ \\

\smallskip
Moreover, if $y\neq y^\prime$ in $\Y_{m-1}$, then $\gcd(h_y,h_{y^\prime})=1,$ hence $\gcd(h_i,h_j^\prime)=1,$ for every factor $h_i$ (res. $h_j^\prime$) of $h_y$ (res. $h_{y^\prime}$). We deduce that
$$\prod_{y\in \Y_{m-1}}h_1^{2e_1-1}\cdots h_{r_y}^{2e_{r_y}-1} \mid F$$
which gives
$$\sum_{y \in \Y_{m-1}}\sum_{i=1}^{r_y}(2e_i-1)\deg(h_i)\leq \deg(F)\leq 3(d-1).$$
\end{proof}

\begin{Rem}
Let $p=\chara(k)$ be the characteristic of the field $k.$ Then there are two cases:
\begin{enumerate}
\item[(i)] Case 1: $p$ does not divide $d.$ Then  $I_{m+1}(J(\ff))\neq 0$ if and only if $[k(\ff):k(\XX)]$ is separable, where $\XX:=X_0,\ldots,X_m.$ In particular, if $p=0,$ then the condition $I_{m+1}(J(\ff))\neq 0$  always holds.
\item[(ii)] Case 2: $p$ divides $d.$ Then  $I_{m+1}(J(\ff))\neq 0$  only if $[k(\ff):k(\XX)]$ is separable.
\end{enumerate} 
\end{Rem}
Note that if $I_{m+1}(J(\ff))\neq 0,$ then $I_j(J(\ff))\neq 0,$ for all $1\leq j\leq m+1.$ In particular, if the characteristic of $k$ is 0, then the assumption $I_3(J(\ff))\neq 0$ is always satisfied.
\begin{Rem}$\;$
\begin{enumerate}
\item[(i)] The inequality 
$$\sum_{y \in \Y_{m-1}}\deg(h_y)\leq \sum_{y \in \Y_{m-1}}\sum_{i=1}^{r_y}(2e_i-1)\deg(h_i)$$
becomes an equality if and only if the defining equation of the unmixed component of the fiber $\pi^{-1}(y)$ has no multiple factors, for every $y\in \Y_{m-1}.$
\item[(ii)] The bound
$$\sum_{y \in \Y_{m-1}}\sum_{i=1}^{r_y}(2e_i-1)\deg(h_i)\leq \deg(F)$$
is optimal as the following example shows.
\end{enumerate} 
\end{Rem}
\begin{Exem} \cite[Example~10]{QHTran17} Let $d\geq 4 $ be an integer. Consider the parameterization given by $\ff=f_0,\ldots,f_3,$ with 
\begin{align*}
\begin{array}{ccc}
f_0= X_0^{d-3}X_1(X_0^2-X_1^2),&&f_2=X_0^{d-3}X_2(X_1^2-X_2^2),\\
f_1=X_0^{d-3}X_2(X_0^2-X_1^2),&&f_3=X_1^{d-3}X_2(X_1^2-X_2^2).
\end{array}
\end{align*}
Using \texttt{Macaulay2} \cite{Macaulay2}, the greatest common divisor  of generators of $I_3(J(\ff))$ is
$$F=X_0^{2d-7}X_2(X_0^2-X_1^2)(X_1^2-X_2^2).$$ 
It is known as in \cite[Example~10]{QHTran17} that
$$\sum_{y\in \Y_1}\deg(h_y)=d+2$$
and
$$\sum_{y \in \Y_1}\sum_{i=1}^{r_y}(2e_i-1)\deg(h_i)=2(d-1)= \deg(F)< 3(d-1).$$
Furthermore, if $d=4,$ then $$\sum_{y\in \Y_1}\deg(h_y)=\sum_{y \in \Y_1}\sum_{i=1}^{r_y}(2e_i-1)\deg(h_i)=\deg(F).$$
\end{Exem}
\section{Bound for the number of one-dimensional fibers of a parameterization surface}
In this section, we will treat  the case of parameterization $\phi:  \PP_k^2 \dashrightarrow \PP_k^3$ of algebraic rational surfaces. Such a map $\phi$ is defined by four homogeneous polynomials $f_0,\ldots,f_3,$ not all zero, of the same degree $d,$ in the standard graded polynomial ring $R=k[X_0,X_1,X_2].$ Our objective is to refine the bound for the cardinality of the set of points in $\PP_k^3$ with a one-dimensional fiber, that is, the cardinality of the set 
$$\Y_1=\{y\in \PP_k^3\, \mid\, \dim \pi^{-1}(y)=1\}.$$

The following result is a direct consequence of Theorem~\ref{Theorem2.5}. It improves the results of \cite{QHTran17} and  the question  \cite[Question~11]{QHTran17} is answered in the affirmative. 
\begin{Cor} \label{Corollary3.1}
Let $I$ be a homogeneous ideal of $R=k[X_0,X_1,X_2]$ generated by a minimal generating set of homogeneous polynomials $\ff:=f_0,\ldots,f_3$ of degree $d.$ Suppose that  $I$ has codimension 2 and $I_3(J(\ff))\neq 0.$ Let $F$ be  the greatest common divisor  of generators of $I_3(J(\ff)).$  Then
$$\sum_{y\in \Y_1}\deg(h_y)\leq \sum_{y\in \Y_1}\sum_{i=1}^{r_y}(2e_i-1)\deg(h_i)\leq \deg(F)\leq 3(d-1),$$
where $h_y=h_1^{e_1}\cdots h_{r_y}^{e_{r_y}}$ is an irreducible factorization of a defining equation  $h_y\in R$ of the unmixed component of the fiber $\pi^{-1}(y),$ for all $y\in \Y_1.$
\end{Cor} 
Now we study the syzygies of $f_i$'s in relation with the degree of the greatest common divisor of the generators of $I_3(J(\ff)).$
 \begin{pro} \label{Proposition2.9}
 Let $\ff:=f_0,\ldots,f_3$ be the homogeneous polynomials of degree $d.$ Let $F$ be  the greatest common divisor  of generators of $I_3(J(\ff)).$ Suppose that $p$ does not divide $d.$ If $\deg(F)=3(d-1)-\delta,$ then  $I=(f_0,\ldots,f_3)$ has a syzygy of degree $\delta:$ there exist the homogeneous polynomials $a_0,\ldots, a_3\in R,$ not all 0, of degree $\delta,$ such that
 $$a_0 f_0+\cdots+a_3f_3=0.$$
 \end{pro}
 \begin{proof}
 If $D_i$ is the $i$-th signed $3\times 3$ of minor of $J(\ff)$, one has  
 $$
 \sum_i D_i \frac{\partial f_i}{\partial X_j}=0
 $$
 for $j=0,1,2$. 
 It then follows from the Euler formula that $(D_0,D_1,D_2,D_3)$ is a syzygy of the $f_i$'s, whenever $d$ is prime to $p$. Set $a_i:=D_i/F$.
 \end{proof}

 \begin{Cor}\label{Corollary2.10}
 Under the assumptions of Proposition~\ref{Proposition2.9}, $\deg(F)=3(d-1)$ if and only if $f_0,\ldots,f_3$ are linearly dependent over $k.$
 \end{Cor}

 \begin{proof}
 Suppose that $\deg(F)=3(d-1).$ By Proposition~\ref{Proposition2.9}, there exist  $a_0,\ldots,a_0\in k,$ not all zero, such that 
 $$a_0f_0+a_1f_1+a_2 f_2+a_3f_3=0.$$ 
 	
 \smallskip
 Suppose that $f_0,\ldots,f_3$ are linearly dependent over $k.$ Then there are $\lambda_0,\ldots, \lambda_3\in k,$ not all 0, such that  $\lambda_0f_0+\cdots+\lambda_3f_3=0.$ Without loss of the generality, we assume that $\lambda_0=-1,$ hence $f_0=\lambda_1f_1+\lambda_2 f_2+ \lambda_3f_3.$ It follows that 
 $$\frac{\partial f_0}{\partial X_j}=\lambda_1\frac{\partial f_1}{\partial X_j}+\lambda_2\frac{\partial f_2}{\partial X_j}+ \lambda_3\frac{\partial f_3}{\partial X_j}, \; \text{for all}\; j=0,1,2.$$
 Thus, we obtain $I_3(J(\ff))=(F_1,\lambda_1 F_1,\lambda_2 F_1,\lambda_3 F_1),$ which shows that $F=F_1.$
\end{proof}
 
We denote by $\Syz(I)\subseteq R^4$ the module of syzygies of $I$. It is a graded module and in the structural graded exact sequence
 $$0\longrightarrow Z_1\longrightarrow R^4(-d) \xrightarrow{(f_0,f_1,f_2,f_3)} I\longrightarrow 0,$$
 we have the identification $\Syz(I)=Z_1(d).$  Recall that for a finitely generated graded $R$-module $M,$ its initial degree is defined by
 $$\indeg(M):=\inf \{ \nu  \mid  M_\nu\neq 0\},$$
 with the convention $\indeg(M)=+\infty$ when $M=0$.
 
 \begin{Cor}\label{Corollary4.4}
 Under the assumptions of Corollary~\ref{Corollary3.1}, if $p$ does not divide $d,$ then 
 $$\sum_{y\in \Y_1}\deg(h_y)\leq 3(d-1)-\indeg(\Syz(I))< 3(d-1),$$
where $h_y=h_1^{e_1}\cdots h_{r_y}^{e_{r_y}}$ is an irreducible factorization of a defining equation $h_y\in R$  of the unmixed component of the fiber $\pi^{-1}(y),$ for all $y\in \Y_1$
 \end{Cor}
 \begin{proof}
 By Proposition~\ref{Proposition2.9},
 \begin{align*}
 \deg(F)\leq 3(d-1)-\indeg(\Syz(I)).
 \end{align*}
 and $\indeg(\Syz(I))=0$ if and only if $f_0,\ldots, f_3$ are linearly dependent over $k.$
 \end{proof}
 Notice that the conditions $I_3(J(\ff))\neq 0$ and $p$ does not divide $d$ are automatically satisfied if $k$ is of characteristic zero.
 
 \begin{Exem}\cite[Example~2]{QHTran17} 
 Consider the parameterization given by $\ff=f_0,f_1,f_2,f_3,$ with
 \begin{align*}
 f_0= X_1^2X_2^4-X_1^4X_2^2,\qquad &f_2=X_0^2X_1^2X_2^2-X_0^2X_1^4,\\
 f_1=X_0^4X_2^2-X_2^6,\qquad \quad\;&f_3=X_0^4X_1^2-X_1^2X_2^4.
 \end{align*}
 Using  {\tt Macaulay2} \cite{Macaulay2}, the greatest common divisor  of generators of $I_3(J(\ff))$ is 
 $$F=X_0X_1^3X_2(X_0^4-X_2^4)(X_1^2-X_2^2).$$
 It is known as in \cite[Example~2]{QHTran17} that
 $$\sum_{y\in \Y_1}\deg(h_y)=8\leq\deg(F)=11\leq 3.5-\indeg(\Syz(\ff))=13.$$
 \end{Exem}

% ====================BIBLIOGRAPHIE =================
%\bibliographystyle{alpha} % style alphabétique en anglais
%\bibliographystyle{plain} % style numérote en anglais
%\bibliography{bibliothese.bib} % pour afficher la biblio
% ============== BIBLIOGRAPHIE=======================

\end{document}